\documentclass{amsproc}
\usepackage[utf8]{inputenc}
  \usepackage{amsmath}
\usepackage{enumerate}
\usepackage{enumitem}
  \usepackage{amssymb}
  \usepackage{amsthm}
  \usepackage{bbm}
  \usepackage{bigints}
  \usepackage{esint}
  \usepackage{hyperref}
  \usepackage{bm}
  \usepackage{bbm}
  \usepackage{graphicx}
  \usepackage{color}
  \usepackage{todonotes}
\usepackage{epsfig}
\usepackage{amsmath,amssymb}
\hoffset - 0.5 in
\usepackage{comment}

\newtheorem{Th}{Theorem}[section]
\newtheorem{lem}[Th]{Lemma}

\theoremstyle{definition}

\newtheorem{Cor}[Th]{Corollary}
\newtheorem{Prop}[Th]{Proposition}

\theoremstyle{remark}

\numberwithin{equation}{section}

\newcommand{\tend}[3][]{\xrightarrow[#2\to#3]{#1}}

\newcommand{\egdef}{\stackrel{\textrm {def}}{=}}
\newcommand{\ds}{\displaystyle}

\newcommand{\R}{\mathbb{R}}

\newcommand{\Z}{\mathbb{Z}}

\newcommand{\N}{\mathbb{N}}
\newcommand{\T}{\mathbb{T}}

\newcommand{\A}{\mathcal{A}}
\newcommand{\Cc}{\mathcal{C}}

\newcommand{\Q}{\mathcal{Q}}

\newcommand{\B}{\mathcal{B}}

\newcommand{\F}{\mathcal{F}}
\newcommand{\Pro}{{\mathbb{P}}}

\newcommand{\mob}{\boldsymbol{\mu}}

\newcommand{\lio}{\boldsymbol{\lambda}}
\newcommand{\Un}{\boldsymbol{1}}
\newtheorem*{thank}{\ \ \ \bf Acknowledgment}
\title[M\"{o}bius random law ]{M\"{o}bius random law and infinite rank-one maps}

\author{e. H. el Abdalaoui}
\address{Normandie University, University of Rouen
  Department of Mathematics, LMRS  UMR 60 85 CNRS\\
Avenue de l'Universit\'e, BP.12
76801 Saint Etienne du Rouvray - France .}
\email{elhoucein.elabdalaoui@univ-rouen.fr}
\urladdr{http://www.univ-rouen.fr/LMRS/Persopage/Elabdalaoui/}
\author{Cesar E. Silva}
\address{Department of Mathematics and Statistics, 18 Hoxsey Street, Williams College, Williamstown, MA 01267 }
\email{csilva@williams.edu }
\urladdr{http://web.williams.edu/Mathematics/csilva/}
\dedicatory{}

\begin{document}
\maketitle
\begin{abstract} We  prove that Sarnak's conjecture holds for any infinite measure symbolic rank-one map. We further extended Bourgain-Sarnak's result, which says that the M\"{o}bius function is a good weight for the ergodic theorem, to maps acting on  $\sigma$-finite measure spaces.  We also discuss and extend Bourgain's theorem by establishing that  there is a class of maps for which the M\"{o}bius disjointness property holds for any continuous bounded function. Our proof allows us to obtain an extension of Bourgain's theorem on M\"{o}bius disjointness for bounded rank one maps and a simple and self-contained proof of this fact. 
\end{abstract}

\section{Introduction}\label{intro}
We are interested in studying the M\"{o}bius-Liouville randomness Law from the dynamical point of view \footnote{The results were announced in \cite{elabdal-Silva} }. 
%The aim of the present note is to contribute to the investigations on the M\"{o}bius-Liouville randomness Law from dynamical view. 
In general terms, this law  states  that the Liouville function and the M\"{o}bius function are orthogonal to any deterministic sequence. In 2010, P. Sarnak in his seminal paper \cite{S} proposed to consider  sequences $(a_n)$ arising from dynamical
systems $(X,T)$, where $X$ is a compact metric space and $T$ is a homeomorphism with topological entropy zero. More precisely, Sarnak made the following conjecture: For any dynamical system $(X,T)$ with topological entropy zero, it is the case that 
$$\frac1{N}\sum_{n=1}^{N}\mob(n)f(T^nx) \tend{N}{+\infty}0,~~~\forall f \in \Cc(X), \forall x \in X,$$
where $\Cc(X)$ is the linear space of all continuous functions on $X$.\\

We recall that the the M\"{o}bius function is  given by $\mob(1)=1$ and
\begin{equation}\label{def:mob}
\mob(n)=
\begin{cases}
(-1)^k& \text{ if $n$ is a product of $k$ distinct primes},\\
0& \text{ otherwise},
\end{cases}
\end{equation}
and the Liouville function $\lio\colon \N^\ast\to \{-1,1\}$ is defined by
$$
\lio(n)=(-1)^{\Omega(n)},
$$
where $\Omega(n)$ is the number of prime factors of $n$ counting multiplicities. The importance of these two functions in number theory is well known and may be illustrated by the following statement
\begin{equation}\label{E:la}
\sum_{n\leq N}\lio(n)={\rm o}(N)=\sum_{n\leq N}\mob(n),
\end{equation}
which is equivalent to the Prime Number Theorem, see e.g.\ \cite[p.~91]{Ap}. We recall also the classical connection of $\mob$ with the Riemann zeta function, namely
$$
\frac1{\zeta(s)}=\sum_{n=1}^{\infty}\frac{\mob(n)}{n^s} \text{ for any }s\in\mathbb{C}\text{ with }\Re(s)>1.
$$
In \cite{Titchmarsh}, it is shown that the Riemann Hypothesis is equivalent to the following: for each $\varepsilon>0$, we have
\[
\sum_{n\leq N}\mob(n)={\rm O}_\varepsilon\left(N^{\frac12+\varepsilon}\right) \text{ as }  N \to \infty.
\]
This latter result is due to Littlewood.\\

In this article, our aim is to investigate the M\"{o}bius-Liouville randomness Law for  dynamical sequences arising from dynamical systems $(X,T)$ where $X$ is now locally compact, and $T$ is a homeomorphism with topological entropy zero; that is, in the same spirit of Sarnak's conjecture, we ask the following: Let $X$ be a locally compact space and $T$ an homeomorphism with topological entropy zero. Do we have
\begin{eqnarray}\label{Sarnak-0}
\frac1{N}\sum_{n=1}^{N}\mob(n)f(T^nx) \tend{N}{+\infty}0,~~~\forall f \in \Cc_0(X), \forall x \in X,
\end{eqnarray}
where $\Cc_0(X)$ is the linear space of all continuous functions on $X$ which vanish at infinity.\\

Our attention is focused on  dynamical systems for which there is no finite invariant measure, or if there is a finite invariant measure it is atomic. 

We start by showing that for the classical example $T:~~n\in \Z \mapsto n+1$, this version of M\"{o}bius disjointness holds. 
Indeed, let $(f(n))_{n \in \Z}$ be a convergent sequence, say to $c$. Then $((f-c)(n))_{n \in \Z}$ converges to zero, that is, it  vanishes at infinity. Therefore by Cesaro's theorem 
\begin{align}
	\frac{1}{N}\sum_{ n \geq N} |f(n)-c| \tend{N}{+\infty}0.
\end{align}
But
\begin{align}
\Big|	\frac{1}{N}\sum_{ n \geq N}\mob(n)(f(n)-c)\Big|
\leq  \frac{1}{N}\sum_{ n \geq N} |f(n)-c|,	
\end{align}
since the M\"{o}bius function is bounded. We thus get 
\begin{align}
	\frac{1}{N}\sum_{ n \geq N}\mob(n)(f(n)-c) \tend{N}{+\infty}0.
\end{align}
 To conclude that $f$ is orthogonal in the sense of Rauzy,  it suffices to notice that, by the Prime Number Theorem, we have 
 \begin{align}
 \frac{1}{N}\sum_{ n \geq N}\mob(n)\tend{N}{+\infty}0.
 \end{align}
 
At this point we have established that  M\"{o}bius disjointness holds at the point 0. We conclude by noticing that the same proof can be run for any $x$ in $\Z.$

We remark  that the algebra of functions that we choose is important. Indeed, for the shift map of $\Z$, all function are continuous. If we consider M\"{o}bius disjointness for all continuous functions, then we can see that orthogonality fails. Indeed, take $f(n)=\mob(n),$ for all $n \in\Z$, (the definition  of $\mob$ is extended in the usual fashion). Therefore,
\begin{align}
\frac{1}{N}\sum_{ n \geq N}\mob(n) f(S^n(0))= \frac{1}{N}\sum_{ n \geq N}\mob(n)^2\tend{N}{+\infty}\frac{6}{\pi^2}.
\end{align}
by the classical computation of the density of the square-free sets.   This simple example shows that for the infinite case  we need to specify the algebra of functions and cannot consider all continuous functions. We also note that if the algebra is the space of periodic functions, then  M\"{o}bius disjointness holds by Dirichlet's theorem. 
This example covers all the dissipative cases as any ergodic invertible dissipative transformation is isomorphic to the shift \cite[Exercise 1.2.1, p.22]{Aa}.

For the conservative case,  we are interested in the class of dynamical system called infinite rank-one, i.e., rank-one transformations with an infinite Radon invariant measure. 

 J. Bourgain initiated the study of Sarnak's conjecture for  finite rank-one maps in \cite{B}, where  he established that the conjecture is true for a class of rank-one maps with bounded parameters. The proof is based on some spectral arguments. We note  that in Bourgain's proof, the fact that the measure is finite or infinite is not used  and therefore the proof is valid for  the infinite measure case; we will extend and present a self-contained proof of Bourgain's argument in a later section.

Subsequently, E. H. el Abdalaoui, Lema\'{n}czyk, and de la Rue extended Bourgain's result to a large class of finite rank-one maps \cite{elabdal-lem-de-la-rue}. Their argument is based on the weak-closure limits in the centralizer of the given rank-one maps. These arguments cannot be extended to the infinite measure case. 
In the same year, V. Ryzhikov gave a simple proof of Bourgain's result for a specific class of finite rank-one maps. Indeed, he established that weakly mixing non-rigid rank-one maps with bounded parameters have  the minimal self-joinings property (MSJ) \cite{Ryzhikov}. Furthermore, it is well-know that the MSJ property combined with the  Katai-Bourgain-Sarnak-Ziegler criterion implies that Sarnak's conjecture holds for the class of maps with the 
minimal self-joinings property. Later, Ryzhikov's result was revisited by A. Gao and H. Hill in \cite{Gao-Hill}, where the authors extended Ryzhikov's result by establishing that canonical bounded rank-one maps have  trivial centralizer.     Despite all these initiatives and efforts, the validity of Sarnak's conjecture for all rank-one maps is still open. We note  again that those arguments cannot be used for the infinite measure case.

Here, our concern is to study this conjecture for infinite symbolic rank-one maps. It  turns out that in this setting we are able to obtain that the conjecture is true (Corollary~\ref{main-2}). We start by proving that \eqref{Sarnak-0} holds for almost all points. 

%We further discuss the problem of convergence almost sure for the Boole map the Hopf ratio pondered with  M\"{o}bius function.

\section {Condition \eqref{Sarnak-0} holds for almost all points.}

In this section we shall extend the Bourgain-Sarnak result which states that if $(X,\B,T,\mu)$ is a dynamical system where  $\mu$ is a probability measure and $T$ is a measure-preserving transformation on $X$, then for almost all points with respect to $\mu$, for any $f \in L^2(X)$, we have
$$\frac1{N}\sum_{n=1}^{N}\mob(n)f(T^nx) \tend{N}{+\infty}0.$$
In  fact,  Sarnak pointed out that his result can obtained as a consequence of the following Davenport's estimate \cite{Da}: for each $A>0$, we have
\begin{equation}\label{vin}
\max_{z \in \T}\left|\displaystyle\sum_{n \leq N}z^n\mob(n)\right|\leq C_A\frac{N}{\log^{A}N}\text{ for some }C_A>0,
\end{equation}
combined with the spectral theorem and some tools from \cite{Bourgain-ergodic}.
Here, we shall prove the following proposition.

\begin{Prop}Let $(X,\B,T,\mu)$ be a $\sigma$-finite conservative dynamical system, where $X$ is a locally compact space and $T$ is
a measure-preserving transformation on $X$ and $\mu$ a $\sigma$-finite measure. Then, for any $f \in \Cc_0(X)$, for almost all
$x$ with respect to $\mu$, we have
 \begin{eqnarray}\label{a-s}
 \frac1{N}\sum_{n=1}^{N}\mob(n)f(T^nx) \tend{N}{+\infty}0.
 \end{eqnarray}
\end{Prop}
\begin{proof}Since the subspace of continuous functions with compact support $\Cc_K(X)$ is dense in $\Cc_0(X)$, it is suffices to
prove \eqref{a-s} for $f \in \Cc_K(X)$. But, if $f$ is $\Cc_K(X)$ then we can apply the spectral theorem to write
\[
\Big\|\frac1{N}\sum_{n\leq N} f(T^nx)\mob (n)\Big\|_2=\Big\|\frac1{N}\sum_{n\leq N} z^n \mob (n)\Big\|_{L^2(\sigma_f)},
\]
where $\sigma_f$ is the spectral measure of $f$. We know  that the spectral measure $\sigma_f$ is a finite  measure on the circle determined by its Fourier transform given by 
\[\widehat{\sigma}_f(n)=\int f\circ T^n\cdot \overline{f}\ d\mu,\] $n\in\Z$.
 This, combined by Davenport's estimate~\eqref{vin} yields, for each $A>0$,
\begin{equation}
  \label{eq:Dav2}
  \Big\|\frac1{N}\sum_{n\leq N} f(T^nx)\mob (n)\Big\|_2 \leq \frac{C_A}{{\log^{A}N}},
\end{equation}
where $C_A$ is a constant that depends only on $A$. Take $\rho>1$, then for $N=[\rho^m]$ for some $m\ge1$, we rewrite \eqref{eq:Dav2} as follows
\[
\Big\|\frac1{N}\sum_{n\leq N}f(T^nx)\mob (n)\Big\|_2 \leq \frac{C_A}{{(m\log(\rho))}^{\varepsilon}}\text{ for any }A>0.
\]
By choosing $A=2$, we obtain
\[
\sum_{ m \geq 1}\Big\|\frac1{[\rho^m]}\sum_{n\leq [\rho^m]}f(T^nx)\mob (n)\Big\|_2 <+\infty.
\]
In particular, by the triangle inequality for the $L^2$ norm,
\[
  \sum_{ m \geq 1} \left| \frac1{[\rho^m]}\sum_{n\leq [\rho^m]}f(T^nx)\mob (n) \right| \in L^2(X,\mathcal{B},\mu)
\]
and the above sum is almost surely finite. Hence, for almost every point $x \in X$, we have
\begin{eqnarray}\label{aj}
\frac1{[\rho^m]}\sum_{n\leq [\rho^m]}f(T^nx) \mob (n) \tend{m}{\infty}0.
\end{eqnarray}
Now, if $[\rho^m]\leq N < {[\rho^{m+1}]+1}$ then we have
\begin{align*}
\Big|\frac1{N}\sum_{n\leq N} f(T^nx) \mob (n)\Big|&=\Big| \frac1{N}\sum_{n\leq [\rho^m]}f(T^nx) \mob (n)+ \frac1{N}\sum_{[\rho^m]+1\leq n\leq N}f(T^nx) \mob (n)\Big|\\
&\leq  \Big| \frac1{[\rho^m]}\sum_{n\leq [\rho^m]}f(T^nx) \mob (n)\Big|+\frac{\|f\|_{\infty}}{[\rho^m]} (N-[\rho^m])\\
&\leq  \Big|\frac1{[\rho^m]}\sum_{n\leq [\rho^m]}f(T^nx) \mob (n)\Big|+\frac{\|f\|_{\infty}}{[\rho^m]} ([\rho^{m+1}]-[\rho^m]).
\end{align*}
Since $\frac{\|f\|_{\infty}}{[\rho^m]} ([\rho^{m+1}]-[\rho^m])\tend{m}{+\infty}\|f\|_{\infty}(\rho-1)$,  using~\eqref{aj} and the fact that $\rho$ can be taken arbitrarily close to 1, we obtain
\[
\frac1{N}\sum_{n\leq N} f(T^nx) \mob (n) \tend{N}{\infty}0 \text{ for a.e. }x\in X,
\]
completing the proof.
\end{proof}
The previous result is  a  Bourgain-Sarnak theorem on the M\"{o}bius disjointness almost everywhere. For the infinite measure case, it is well known that for all $f \in L^1(X)$, the classical Bikhoff ergodic averages converge to $0$ a.e.. It follows that, for all  $f \in L^1(X)$, for almost all $x \in X$,
\[
\frac1{N}\sum_{n\leq N} f(T^nx) \mob (n) \tend{N}{\infty}0 .
\]
However, the natural way to formulate  the M\"{o}bius disjointness almost everywhere in the case of infinite measure is to  consider the Hopf ratio ergodic averages pondered with  M\"{o}bius function, that is,
\begin{align}\label{Hopf}
	\frac{\displaystyle \sum_{n=1}^{N}\mu(n)f(T^nx)}{\displaystyle\sum_{n=1}^{N}p(T^nx)}.
\end{align}
\noindent{}where $f,p \in L^1(X,\nu)$, $p>0$, and $T$ is conservative measure preserving transformation of the $\sigma$-finite measure space $(X,\mathcal{A},\nu)$.

We notice that Hopf's maximal ergodic inequality holds for \ref{Hopf}. Also, if $T$ is an ergodic measure preserving transformation of a $\sigma$-finite measure space $(X,\A,\nu)$, then , for any $f, p\in L^1(X), p>0,$ for almost all $x \in X$,
\begin{align}\label{Hopf-sup}
\limsup_{N\to \infty}\Bigg|\frac{\displaystyle \sum_{n=1}^{N}\mu(n)f(T^nx))}{\displaystyle\sum_{n=1}^{N}p(T^nx)}\Bigg|
\leq \frac{\displaystyle \int |f(y)| d\nu(y)}{\displaystyle \int p(y) d\nu(y)}.
\end{align}
%\vskip 0.5cm\todoH{Plz I add that the $\limsup$ is bounded by $\frac{\int f d\nu}{\int p d\nu }$}

Now,  for a dissipative map $T$, for any ${f},{p}\in{L}^{1}{\left({X},\mu\right)}$, for almost all ${x}\in{X}$, we have

\[
\lim_{{{n}\rightarrow+\infty}}{\frac{{{{\sum_{{n}={1}}^{N}}\mu{\left({n}\right)}{f{{\left({T}^{n}{x}\right)}}}}}}{{{{\sum_{{n}={1}}^{N}}{p}{\left({T}^{n}{x}\right)}}}}}={\frac{{{{\sum_{{n}={1}}^{+\infty}}\mu{\left({n}\right)}{f{{\left({T}^{n}{x}\right)}}}}}}{{{{\sum_{{n}={1}}^{+\infty}}{p}{\left({T}^{n}{x}\right)}}}}}.
\]
However, we show that for the Boole map and $ {f{\in}}{L}^{\infty}{\left({X},\mu\right)}$   the convergence does not hold.
The Boole map is given by $T:  x \in \R\setminus\{0\} \mapsto x-\frac{1}{x}$. It is well known  that $T$  preserves Lebesgue measure and  is ergodic and conservative \cite[p.215]{Aa}.  Now, taking $f=\mathbbm{1}$, using a deep result of Odlyzko-te Riele \cite{od-te} combined with    \cite[Exercise 2.2.2]{Aa}  we have 
\[
{\overline{\lim}}{\frac{{{{\sum_{{n}={1}}^{N}}\mu{\left({n}\right)}{\mathbbm{1}{{\left({T}^{n}{x}\right)}}}}}}{{{{\sum_{{n}={1}}^{N}}{p}{\left({T}^{n}{x}\right)}}}}}>{\frac{{{1.06 \pi}}}{{{\sqrt{2}\int{p}{\left({y}\right)}{d}\mu{\left({y}\right)}}}}}
\]
and
\[
{\underline{\lim}}{\frac{{{{\sum_{{n}={1}}^{N}}\mu{\left({n}\right)}{\mathbbm{1}{{\left({T}^{n}{x}\right)}}}}}}{{{{\sum_{{n}={1}}^{N}}{p}{\left({T}^{n}{x}\right)}}}}}<{\frac{{-{1.009 \pi}}}{{{\sqrt{2}\int{p}{\left({y}\right)}{d}\mu{\left({y}\right)}}}}}.
\]

\section{Rank-one maps and our first main result}\label{rk-one}
As mentioned earlier, the study of Sarnak's conjecture for the class of symbolic rank-one maps  was initiated by Bourgain \cite{B}, where the author proved that Sarnak's conjecture is true for the subclass of maps for which the cutting parameter $(p(n)) \subset \N$ and
the stacking parameter $(s(n,i))_{i=0}^{p_n-1} \subset  \N^{p_n}, n\in \N$ are bounded and $s(n,p_n-1)=0$. In \cite{elabdal-lem-de-la-rue},
the authors relaxed this latter condition and extended Bourgain's result to a large class of rank-one maps including rigid Generalized Chacon's maps and Katok's maps. However, as it was pointed out, their methods can not be extended to the case of infinite measure.\\

 There are several definitions of rank-one maps in ergodic theory. Here we consider the symbolic one. Let $T$ be a shift map
 on $\{0,1\}^{\Z}$ and $(p(n))_{n \in \N}$  be a positive sequence of integers. For each $n$, consider the $p_n$-tuple of non-negative integer denoted by $(s(n,i))_{i=0}^{p_n-1}$. The sequence $(p(n))_{n \in \N}$ will play the role of what are called the cutting parameters and $((s(n,i))_{i=0}^{p_n-1})_{n \in \N}$ are the spacers parameters. Put
 $$ W_0\egdef 0 ; \quad W_{n+1}\egdef W_n 1^{s_{n,0}}W_n 1^{s_{n,1}}\cdots W_n 1^{s_{n,p_n-1}}. $$
 The sequence of finite words $(W_n)_{n\ge 1}$ over the alphabet $\{0,1\}$ is called the \emph{building blocks} sequence.\\

 The length $|B_n|$ of the building block of order $n$ is equal to $h_n$ and the symbols $1$ in the building blocks will  be called the \emph{spacers}.\\

  Then we consider the subshift $X \subset \{0,1\}^\Z$ consisting of  the set of bi-infinite sequences $(x_j)_{j\in\Z}$ satisfying
  \begin{equation}
    \label{eq:def_subshift}
    \forall i<j ,\ x|_i^j\egdef x_i x_{i+1} \ldots x_{j-1} \mbox{ is a subword of }W_m\mbox{ for some $m\ge1$}.
  \end{equation}
We consider on $X$ the product topology which turns $X$ into a metrizable compact space, and we denote by $T_s$ the shift of coordinates, which is a homeomorphism of $X$. According to Kalikow \cite{Kalikow}, $T_s$ is said to be \emph{nondegenerate} if the limit sequence $W_\infty\egdef \lim_{n\to\infty}W_n$ is aperiodic. Otherwise, $T_s$ is said to be \emph{degenerate}. We further have that if  $T_s$ is degenerate, then $T_s$ is isomorphic to an odometer, in particular, it has infinitely many rational eigenvalues. For the proof of this fact and more details on the connection between the classical and useful definition of rank-one map by cutting and staking method and the symbolic definition, we refer to \cite{elabdal-lem-de-la-rue}.\\

Assume that $T_s$ is nondegenerate and let $v$ be a finite word. For $k \in \Z$, put
$$O_{v,k}=\big\{x \in X: x {\textrm{~~has~~an~~occurrence~~of~~}}v~~\textrm{at~~position}~~k\big\}.$$
It is follows that $O_{v,k}$ is a basic open set of $X$. We further have that there is an atomless shift-invariant measure $\mu$
on X given by
$$\mu(O_{v,k})=\lim_{n \to \infty}\frac{fr(v,W_n)}{fr(0,W_n)},$$
where $fr(u,W_n)$ is the number of occurrences of $u$ in $W_n$. We notice that $\mu$ is the unique shift-invariant measure on $X$ with $\mu(O_{0,0})=1$. $T_s$ is said to be infinite if $\mu(O_{1,0})=\infty.$ In this case, $\mu$ is the unique invariant measure up to constants. There is also the atomic measure concentrated on the infinite word $\Un=1111\cdots.$\\

We are now able to state our first main result.
%\begin{Th}\label{main} The M\"{o}bius disjointness law holds for any infinite rank-one map.
%\end{Th}
\begin{Th}\label{main-1}  Let $(X,T)$ be a dynamical system ($X$ is a compact metric space and $T$ a homeomorphism). 
If $(X,T)$ has only one invariant probability measure and this measure is atomic, then it satisfies the
M\"{o}bius disjointness law.
\end{Th}
For the proof of Theorem \ref{main-1} we need the following lemma.
\begin{lem}Let $(X,T)$  be as in Theorem~\ref{main-1} and let $z\in X$ be the point where the finite invariant measure is supported. Then for any $f \in \Cc(X)$, for any $x \in X$,
we have
$$\frac1{N} \sum_{n=1}^{N} \big(f-f(z)\big)(T^nx) \to 0.$$
 \end{lem}
\begin{proof} By a standard argument, the set $\big\{\frac1{N} \sum_{n=1}^{N} \delta_{T^nx}\big\}$ is compact for the weak convergence star topology. Therefore, for some sequence $(N_k)$ the sequence of probability measures
$$\frac1{N_k} \sum_{n=1}^{N_k} \delta_{T^nx}$$
converges to some measure $\nu$, that is, for any function  $f \in \Cc(X)$, we have
$$\frac1{N_k} \sum_{n=1}^{N_k} f(T^nx) \to \nu(f).$$
We thus get
$$\frac1{N_k} \sum_{n=1}^{N_k} \big(f-f(z)\big)(T^nx) \to \nu(f)-f(z).$$
But the only probability invariant measure is $\delta_{z}$. Hence
$$\frac1{N_k} \sum_{n=1}^{N_k} \big(f-f(z)\big)(T^nx) \to 0.$$
This complete the proof of the lemma.
\end{proof}
\begin{proof}[\textbf{Proof of Theorem \ref{main-1}.}] Let $f$ be a continuous function and write
$$\frac1{N}\sum_{n\leq N} f(T^nx) \mob (n)=\frac1{N}\sum_{n\leq N} \big(f-f(z)\big)(T^nx) \mob (n)+f(z).\frac1{N}\sum_{n\leq N}\mob (n).$$
Then the first term of the right-hand side is zero by the lemma and the second by the Prime Number Theorem. This finish the proof of the theorem.
\end{proof}

\begin{Cor}\label{main-2} The M\"{o}bius disjointness law holds for any infinite symbolic rank-one map.
\end{Cor}
\begin{proof} Let $T$ be a rank-one map as defined in Section \ref{rk-one}. It is infinite measure-preserving and it only has one finite invariant measure supported on $\Un$. Then apply Theorem~\ref{main-1}.
\end{proof}

Our first main result (Theorem \ref{main-1}) can also be applied  to the examples from the infinite $(C,F)$-actions constructed  by A. Danilenko \cite[see Theorem 2.1 and the remark that follows]{Sasha}.
%%%%
\section{Our second main result, Riesz products as spectral types and Bourgain's arguments}\label{bourgain_sec}

In this section we extend and give a self-contained presentation of Bourgain's argument and use it to obtain M\"{o}bius disjointness for a larger class of rank-one maps including rank-one maps acting on $\sigma$-finite space. In fact, We  state and prove our  main result about rank-one maps, which is an extension of Bourgain's theorem on the M\"obius disjointness for the class of rank-one maps with bounded parameters.

\begin{Th}\label{Bourgain}Let $T=T_{((p_n),{(s(n,i))}_{i=0}^{p_n-1},{n \in \N})}$ be a symbolic rank-one map and assume that 
	\[
	\sum_{n=1}^{+\infty}\frac1{p_n^2}=+\infty.
	\]
	Then the M\"{o}bius disjointness law holds for $T$.
\end{Th}
Theorem \ref{Bourgain} extends in some sense Bourgain's theorem. For its proof, we will follow Bourgain's method. We will assume also that $T$ is weak-mixing and there are infinitely  set of primes $q$ such that $q$ does not divide $\Big\{(p_n-1)h_n+\ds \sum_{j=0}^{p_n-1}s(n,j)\Big\}$.

%%%%%%

%%%%%%%%
%%%%%%%
%\begin{comment}

%\section{Riesz products as spectral types and Bourgain's arguments}\label{bourgain_sec}

%The purpose of this section is to give  a simple and self-contained exposition of Bourgain's proof of the M\"obius disjointness for the class of rank-one maps with bounded parameters. It turns out that our proof allows us to  also obtain  an extension of Bourgain's theorem. As a consequence, we obtain a unified proof of three results on the M\"obius disjointness in the class the rank-one maps.  M\"obius disjointness is connected to Sarnak conjecture \cite{S} and has been introduced in the context of rank-one maps by Bourgain \cite{B}. We will thus present essentially Bourgain's spectral approach.\\

We start by pointing out that it is enough to verify that M\"obius orthogonality holds for the dense subset of functions. Indeed, we have the following proposition.

\begin{Prop} Let $X$ be a compact set and $\F$ be a family of continuous functions on $X$. Assume that
	\begin{enumerate}[label=(\alph*)]
		\item $\F$ is dense in $C(X)$.
		\item For any $f \in \F$, for any $x \in X$, the sequence $\bigl(f(T^nx)\bigr)_{n\geq1}$ is orthogonal to the M\"obius function.
	\end{enumerate}
	Then, for any $f \in C(X)$, for any $x \in X$, the sequence $\bigl(f(T^nx)\bigr)_{n\geq1}$ is orthogonal to the M\"obius function.
\end{Prop}
\begin{proof}Straightforward.
\end{proof}

Now, we will recall some basic facts on the notion of affinity. This notion is at the heart of Bourgain's approach.

\paragraph{\bf The affinity between two probabilities.}
Let $\Pro$ and $\Q$ be probability measures on measurable space $(X,\A)$. The affinity or Hellinger integral $H(\Pro,\Q)$ is defined by the formula
\[
\bigintss_{X}\sqrt{\frac{d\Pro}{d\R}.\frac{d\Q}{d\R}} d\R.
\]
The most important property of affinity needed it here is the following.

\[
H(\Pro,\Q)=0 \Longleftrightarrow \Pro \perp \Q.
\]
We shall need also the following classical lemma due to Coquet-Mend\'es-France-Kamae \cite{CMFK}.
\begin{lem}\label{CoquetMFK}
	Let $(\Pro_n)$ and $(\Q_n)$ be two sequences of probability measures
	on the circle weakly converging to the probability measures $\Pro$ and $\Q$ respectively. Then
	\begin{equation}\label{limsup}
	\limsup_{n \longrightarrow +\infty } H(\Pro_n,\Q_n) \leq H(\Pro,\Q).
	\end{equation}
\end{lem}

\paragraph{\bf Spectral measures.}
Let $(X,\A,\Pro,T)$ ba a ergodic dynamical system. It is well known that the spectral measure of any $L^2$-function $f$ can be defined as the weak limit of the following sequence of finite measures
\[
\sigma_{f,T,N}=\Big|\frac{1}{\sqrt{N}}\sum_{j=0}^{N-1}f(T^nx) e^{i n \theta}\Big|^2 d\theta.
\]
The weak limit $\sigma_f$ holds for almost all points $x$. If in addition $X$ is compact and the system $(X,\A,\Pro,T)$ is uniquely ergodic, then the weak limit holds for any $x$ provided that the function $f$ is continuous.\\

We denote by $\sigma_p$ the push-forward measure of $\sigma$ under the map $z \mapsto z^p$.\\

We introduce also the notion of the pseudo-dilation of a measure. It is an easy exercise to see that for any $\phi$ in $L^1(\T)$ and $m \in \Z^*$,
if $\phi_{(m)}(t)=\phi(mt)$, for any $t \in \T$, then, for any $n \in  \Z$, we have
\[
\widehat{\phi_{(m)}}(n)=\left\{
\begin{array}{ll}
0, & \hbox{if $n \nmid  m$;} \\
\widehat{\phi}(\frac{n}{m}), & \hbox{if $n | m$.}
\end{array}
\right.
\]
We extend this fact to the finite measure $\sigma$ on the torus by putting
\[
\widehat{\sigma_{(m)}}(n)=\left\{
\begin{array}{ll}
0, & \hbox{if $n \nmid m$;} \\
\widehat{\sigma}(\frac{n}{m}), & \hbox{if $n | m$.}
\end{array}
\right.
\]
Therefore we have the following lemma \cite[p.7]{Kat},
\begin{lem}\cite{Kat}\label{dilation}Let $\sigma_k= f_k(t) dt$ be a sequence of finite measure on the torus which converges weakly to
	$\sigma$. Then $\sigma_{(m)}$ is the weak limit of $\sigma_{k,(m)}=f_k(mt) dt$
\end{lem}
From this we deduce the following proposition. 
\begin{Prop}Let $(X,\A,\Pro,T)$ be a uniquely ergodic system, $f$ be a continuous function and $p$ be a positive integer. Then, the sequence of measures given by
	\[
	\sigma_{f,(p),N}= \Big|\frac{1}{\sqrt{N}}\sum_{j=0}^{N-1}f(T^nx) e^{i p n t}\Big|^2 dt,
	\]
	converges to the pseudo-dilation of the spectral measure of $f$ by $p$, that is, $\sigma_{f, (p)}$.
\end{Prop}
The measure $\sigma_{(p)}$ can be obtained by taking the image of $\sigma$ under the map $x\mapsto\frac1px$, i.e.\ the measure $\sigma_{1/p}$, and then repeating this new measure periodically in intervals $[\frac jp,\frac{j+1}p)$, that is:
$$
\sigma_{(p)}:=\frac1p\sum_{j=0}^{p-1}\left(\tau_{j/p}\right)_\ast\sigma_{1/p}=
\frac1p\sum_{j=0}^{p-1}\sigma_{1/p}\ast\delta_{j/p},$$
where $\tau_y(x)=x+y$. Clearly, $\sigma_{(p)}$ is invariant under  rotation by $1/p$, i.e.
\begin{align}\label{zi1a}
(\tau_{1/p})_\ast\sigma_{(p)}=\sigma_{(p)}.
\end{align}
The following two relations follow directly (see \cite{Qu}, p.\ 196):
\begin{align}\label{zi2}
\left(\sigma_{(p)}\right)_p=\sigma,
\end{align}
\begin{align}\label{zi3}
\left(\sigma_{p}\right)_{(p)}=\frac1p\sum_{j=0}^{p-1}\sigma\ast\delta_{j/p}.
\end{align}
%We shall need also the following lemma.
%\begin{Lemma}Let $\sigma$ and $\sigma'$ be a two finite measures on the torus such that
%$\sigma$ is absolutely continuous with respect to $\sigma'$ ($\sigma  << \sigma')$.
%Then $\sigma_{(p)} <<  \sigma'_{(p)}$, for any $p\in \Z^*$.
%\end{Lemma}
We further have the following lemma due to Jean-Paul Thouvenot.

\begin{lem}
	Assume that $\sigma$ and $\eta$ are two probability measures on the circle. 
	\begin{enumerate}[label=(\roman*)] 
		\item  If $\sigma_p\perp\eta_q$, then $\sigma_{(p)}\perp \eta_{(q)}$; \label{Bi}
		\item Assume $(p,q)=1$. Then $\sigma_p\perp \eta_q$ if and only if   $\sigma_{(p)}\perp\eta_{(q)}$ .
		\label{Bj}
	\end{enumerate}
\end{lem}
\begin{proof}The first part \eqref{Bi} follows easily form \eqref{zi2}.
	Indeed, this gives
	$$(\sigma_{(p)})_{pq}=\sigma_q \;\; \textrm{and} \;\; (\sigma_{(q)})_{pq}=\sigma_p.$$
	For the second part \ref{Bj},  Suppose that $(p,q)=1$ and $\sigma_{p})\not\perp\sigma_{q}$. Then
	\begin{align}\label{zi4a}
	\left(\sigma_{p})\right)_{\frac1{pq}}
	\not\perp\left(\eta_{q})\right)_{\frac1{pq}}.
	\end{align}
	If $A\subset[0,\frac1{pq})$ is Borel then
	$$
	\left(\sigma_p\right)_{\frac1{pq}}(A)=\sigma_p(pqA)=\sigma_p(p(qA))=$$
	$$
	\sigma\left(\bigcup_{j=0}^{p-1}\left(qA+\frac jp\right)\right)=
	\sum_{j=0}^{p-1}\sigma\left(qA+\frac jp\right)$$
	as $qA\subset[0,\frac1p)$. We claim that
	\begin{align}\label{zi6}
	\left(\sigma_{p}\right)_{\frac1{pq}}\ll\frac1p\sum_{j=0}^{p-1}\sigma_{(q)}\ast\delta_{\frac j{pq}}.\end{align}
	Indeed, since the support of $\left(\sigma_{(p)}\right)_{\frac1{pq}}$ is contained in $[0,\frac1{pq})$, fix
	a Borel set $A\subset[0,\frac1{pq})$ and suppose that
	$\sigma_q\ast\delta_{\frac j{pq}}(A)=0$. Then
	$$
	\sigma_q\ast\delta_{\frac j{pq}}(A)=\frac1q\sum_{k=0}^{q-1}\sigma^{1/q}\left(\left(A+\frac j{pq}\right)+\frac kq\right)=
	\sigma(qA+\frac jp)$$
	whence \eqref{zi6} follows. In view of \eqref{zi4a} and \eqref{zi6}, we obtain
	\begin{align}\label{zix}
	\sigma_q\ast\delta_{\frac j{pq}}\not \perp
	\eta_p\ast\delta_{\frac k{pq}}
	\end{align}
	for some $0\leq j\leq p-1$ and $0\leq k\leq q-1$. Since $(p,q)=1$, we have $j=ap+bq$ and $k=cp+dq$. Substituting this to~\eqref{zix}, we obtain
	$$
	\sigma_{(q)}\ast\delta_{\frac a{q}}\ast\delta_{\frac bp}\not \perp
	\eta_{(p)}\ast\delta_{\frac cq}\ast\delta_{\frac dp}.$$ Hence
	$$
	\sigma_{(q)}\ast\delta_{\frac {a-c}{q}}\not \perp
	\eta_{(p)}\ast\delta_{\frac {d-b}p}$$
	and the result follows from \eqref{zi1a}. For the converse, let $A,B\subset [0,1)$ be measurable sets with $A\cap B=\emptyset$ and $\sigma_{(p)}(A)=\eta_{(q)}(B)=1$. We may assume that $A=\bigcup_{i=0}^{q-1}(A'+\frac{i}{q})$ and $B=\bigcup_{j=0}^{p-1}(B'+\frac{j}{p})$, where $A'\subset [0,\frac{1}{q})$, $B'\subset [0,\frac{1}{p})$. Let $A'':=qA'$, $B'':=pB'$. Then $\sigma(A'')=\eta(B'')=1$. We claim that
	\begin{equation}\label{a:1}
	pA''\cap(qB''+k)=\emptyset \text{ for any }k\in\Z
	\end{equation}
	(we treat now $A''$, $B''$ as subsets of $\R$). By the definition of $\sigma_{(q)}$, $\eta_{(p)}$, we have $(A'+\frac{i}{q})\cap(B'+\frac{j}{p})$ for all $i,j\in\Z$. Since $(p,q)=1$, for any $\ell\in\Z$ we can find $i,j\in\Z$ such  that $\frac{j}{p}-\frac{i}{q}=\frac{jq-ip}{pq}=\frac{\ell}{pq}$. Therefore, $A'\cap(B'+\frac{\ell}{pq})=\emptyset$ for all $\ell\in\Z$ and it follows that  $pA''\cap (qB''+\ell)=pq A'\cap (pq B'+\ell)=\emptyset$ for all $\ell\in\Z$, i.e.\ \eqref{a:1} indeed holds. It remains to show that \eqref{a:1} implies $\sigma_p\perp\eta_q$. Let $A''':=\bigcup_{i=0}^{p-1}(pA''-i)\cap [0,1)$, $B''':=\bigcup_{j=0}^{q-1}(qB''-j)\cap [0,1)$. Clearly, $\sigma_p(A''')=\eta_q(B''')=1$ and thus, by~\eqref{a:1}, $\sigma_p\perp \eta_q$.
\end{proof}

Let us recall that the maximal spectral type
of $T$ is the equivalence class of Borel measures $\sigma$ on $S^1$
(under the equivalence relation $\mu_1=\mu_2$ if  and only if $\mu_1<<\mu_2$
and $\mu_2<<\mu_1$), such that
$\sigma_f<<\sigma$ for all $f\in L^2(X)$ and
if $\nu$ is
another measure for which $\sigma_f<<\nu$
for all $f\in L^2(X)$ then $\sigma << \nu$.

By the canonical decomposition of $L^2(X)$ into decreasing cycles (see for instance
\cite{P}) with respect to the operator $U_T(f)=f\circ T$,
there exists a Borel measure $\sigma=\sigma_f$ for some $f\in L^2(X)$, such that $\sigma$ is in
the equivalence class defining the maximal spectral type of $T$. By abuse of
notation, we will call this measure the maximal spectral type measure, but
it can be replaced by any other measure in its equivalence class.
As shown by B. Host, F. Mela, and F. Parreau \cite{Host-mela-parreau}, Bourgain \cite{BIsr}, \cite{CN} and \cite{KR},
the maximal spectral type (except possibly some discrete part) of any rank-one map is given by some kind of generalized Riesz products. Precisely, we have the following:

\begin{Prop}\label{rankone-sp} Let $(X,\A,T,\Pro)$ be a rank-one dynamical system with the parameters
	$(p_n)_{n \in \N}$ and $((s_{n,j})_{j=1}^{p_n})_{n \in \N}$. Then the maximal spectral type of $T$ is given (up to discrete part) by the weak limit of the measures $\sigma_N$,$ N = 1, 2, \cdots$ given by the formula
	\[d\sigma_N=\prod_{n=1}^{N}|P_n(t)|^2 dt,\]
	where
	\[
	P_n(t)=\frac1{\sqrt{p_n}}\sum_{j=0}^{p_n-1}{e^{i(jh_n+\widetilde{s}_{n,j})t}}.
	\]
\end{Prop}
The proof of Proposition \ref{rankone-sp} yields that the spectral type of the rank-one  map is the spectral measure of the indicator function of the cylinder set [0](up to discrete part). For any $n$, put
\[
f_n=\frac{1}{\sqrt{\mu([B_n]}}\mathbbm{1}_{[B_n]},
\]
where $[B_n]$ is the cylinder set. The span $F_n$ generated by the functions ${f_0,\cdots,f_n}$ verify
\begin{enumerate}
	\item $F_n \subset  F_{n+1}$.
	\item Any continuous $\phi$ which depend on finitely many coordinates can be approximated by the functions in $\bigcup F_n$.
\end{enumerate}
Regarding Proposition \ref{rankone-sp}, it is sufficient to establish that  M\"obius orthogonality holds only for any function $f_n$. Notice that the pseudo-dilation measure $\sigma_{(p)}$ is the weak-star limit of 
\[d\sigma_{N,(p)}=\prod_{n=1}^{N}|P_n(pt)|^2 dt,\]
by Lemma \ref{dilation}.

We further need the following lemma   due to Nadkarni \cite[p.154]{Nad}.
\begin{lem}\label{Nad-infinite)} The infinite product
	$$\prod_{l=1}^{+\infty}\big|P_{j_l}(t)\big|^2,$$
	taken over a subsequence $j_1<j_2<\cdots$, represents the maximal spectral type (up to a discrete measure) of some rank-one map. If $j_l \neq l$ for infinitely many $l$, the map acts on an infinite measure space (it is an infinite rank-one map.),   
\end{lem}

\section{\bf M\"obius disjointness and Bourgain methods}
The Bourgain methods introduced in \cite{B} are based essentially on the Daboussi-Katai-Bourgain-Sarnak-Ziegler Criterium \cite{BSZ} which  can be formulated as follows.

\begin{Th}[D-K-B-S-Z criterion~\cite{BSZ}]\label{KDSBZ}
	Let $(X,T)$ be a flow ($X$ is a compact and $T$ homeomorphism). Let $f$ be a continuous function on $X$ and assume that for  large prime
	numbers $p$ and $q$, $p\neq q$,  we have, for any $x \in X$,
	\[
	\frac1{N}\sum_{j=0}^{N-1}f(T^{pn}x)f(T^{qn}x) \tend{N}{+\infty}0.
	\]
	Then
	\[
	\frac1{N}\sum_{j=1}^{N}\mu(n)f(T^{n}x) \tend{N}{+\infty}0.
	\]
\end{Th}
It was observed by el Abdalaoui \& M. Nerurkar that KDBSZ \ref{KDSBZ} is essentially based on the Prime Number Theorem. Precisely, the authors noticed that the PNT implies Daboussi-Katai-Bourgain-Sarnak-Ziegler criterion. But, the converse implication is false. 

Now, let $p \neq q$ be a prime numbers. By Bourgain's observation, we have
\[
\Big| \frac1{\widetilde{N}}\sum_{j=1}^{\widetilde{N}}f(T^{pn}x)f(T^{qn}x) \Big| \leq
\bigintss_{\T}  \Big|\frac{1}{\sqrt{\widetilde{N}}}\sum_{j=1}^{N}f(T^{n}x) e^{i pn \theta}\Big|
\Big|\frac{1}{\sqrt{\widetilde{N}}}\sum_{j=1}^{N}f(T^{n}x) e^{i q n \theta}\Big| d\theta,
\]
where $\widetilde{N}=\frac{N}{\max{\{p,q\}}}$. Therefore,
\[
\Big| \frac1{\widetilde{N}}\sum_{j=1}^{\widetilde{N}}f(T^{pn}x)f(T^{qn}x) \Big| \leq
\bigintss_{\T}  \Big|\frac{1}{\sqrt{\widetilde{N}}}\sum_{j=1}^{N}f(T^{n}x) e^{i pn \theta}\Big|
\Big|\frac{1}{\sqrt{\widetilde{N}}}\sum_{j=1}^{N}f(T^{n}x) e^{i q n \theta}\Big| d\theta,
\]
Hence,
\[
\Big| \frac1{\widetilde{N}}\sum_{j=1}^{\widetilde{N}}f(T^{pn}x)f(T^{qn}x) \Big| \leq
\frac{N}{\widetilde{N}}H(\sigma_{f,(p),N},\sigma_{f,(q),N})
\]
This combined with the Coquet-Mand\'es-France-Kamae Lemma \ref{CoquetMFK} implies
\begin{align}
\limsup\Big| \frac1{\widetilde{N}}\sum_{j=1}^{\widetilde{N}}f(T^{pn}x)f(T^{qn}x) \Big| &\leq
\max\{p,q\}\limsup H(\sigma_{f,(p),N},\sigma_{f,(q),N})\\
&\leq \max\{p,q\} H(\sigma_{f,(p)},\sigma_{f,(q)}).
\end{align}
Therefore we are reduced to the computation of the Hellinger distance between two primes pseudo-dilatation 
of the spectral measure $f$. For that, we shall need to extend some basic facts known in the context of generalized Riesz products.

\section{On the orthogonality of the generalized Riesz products}

The orthogonality of two given Riesz products has been intensively studied since the famous Zymund dichotomy criterion  \cite{Zygmund} (see \cite{elabdal-Nad} for a recent exposition). Peryri\'ere \cite{ Peyriere} was the first to extend the Zygumd criterion of singularity to the case of two Riesz products based on the some set of frequencies. The analogous results have been extended to generalized Riesz products \cite{KS},\cite{Parreau},\cite{Brown}, \cite{Brown-Moran},\cite{Ritter} but under the restriction that the frequencies of the polynomials be dissociate. However, the polynomials arising from the rank-one constructions do not enjoy this property, making their treatment more delicate. Nevertheless, using Peyri\'ere's idea, Klemes and Reinhold  extended the Zygmund criterion in this context. One can extend in a similar way Peyri\'ere's criterion to the generalized Riesz products arising from the rank-one spectral type. According to the previous section, this latter criterion can be used to establish  M\"obius disjointness. Here, we need  the following version of Peyri\`ere criterion as it is stated by G. Brown and E. Hewitt \cite{Brown-Hewitt}.
\begin{lem}\label{Peyr}
	Let $(X,\mathcal{M})$ be a measurable space. Let $\mu_1$ and
	$\mu_2$ be a probability measures
	on it. Suppose that
	there is a sequence $(\phi_n)_{n \in \N}$
	of complex-valued functions
	on X that are in $\mathcal{L}_2(X, \mathcal{M}, \mu_j)$
	$(j=1, 2)$ and for which the following relations hold
	\begin{align}
	&\sum_{ k \geq 0} \sup_{n \in \N^*}\Big\{\Big|\int \phi_{n}(x) \phi_{n+k}(x) d\mu_j- 
	\int \phi_{n}(x) d\mu_j \int \phi_{n+k}(x) d\mu_j\Big|\Big\}<\infty. \label{Br1}\\
	& \sum_{ n \geq 1}\Big|\int \phi_{n}(x)  d\mu_1-\int \phi_{n}(x)  d\mu_2\Big|^2=+\infty. \label{Br2}
	\end{align}
	Then, $\mu_1$ and $\mu_2$ are are mutually singular. 
\end{lem}
We need also the following lemma from \cite{KR}.

\begin{lem}\label{Karine}Let $(n_j)_{j=0}^{\infty}$ be a sequence satisfying $n_{j+1} \geq n_j+3$ and $d\alpha=\ds \prod_{j=0}^{\infty}\big|P_{n_j}\big|^2$. Then there is a sequence $(m_j)_{j \in \N}$ such that:
	\begin{enumerate}[label=(\roman*)]
		\item $\widehat{\alpha}(\pm m_j)=\frac{1}{p_{n_j}},$
		\item $\widehat{\alpha}(m_j\pm m_k)=\widehat{\alpha}(\pm m_j) \widehat{\alpha}(\pm m_k)$, $\forall j \neq k$. 
	\end{enumerate}
\end{lem}
The sequence $(m_j)$ is given exactly by 
$$m_j=h_{n_{j+1}}-h_{n_j}-s(n_j,p_{n_j}).$$ 
We are now able to give the proof of Theorem \ref{Bourgain}.
%\todoH{Plz, I add the proof of Theorem \ref{Bourgain} here.}
\noindent \begin{proof}[\textbf{Proof of Theorem \ref{Bourgain}}] By our assumption 
	\[
	\sum_{n=1}^{+\infty}\frac1{p_n^2}=+\infty.
	\]
	Therefore  there is $\eta \in \{0,1,2\}$ such that 
	
	\[
	\sum_{n \equiv \eta \mod 3}\frac1{p_n^2}=+\infty.
	\]
	Let us put $n_j=3j+\eta$. Then $n_{j+1}=n_j+3$. Now, by appealing to Lemma \ref{Karine}, the generalized Riesz product 
	$$d\alpha=\ds \prod_{j=0}^{\infty}\big|P_{n_j}\big|^2,$$
	satisfy  
	\begin{enumerate}
		\item $\widehat{\alpha}(\pm m_j)=\frac{1}{p_{n_j}},$
		\item $\widehat{\alpha}(m_j\pm m_k)=\widehat{\alpha}(m_j) \widehat{\alpha}(m_k)$, $\forall j \neq k$. 
	\end{enumerate}
	where $m_j=h_{n_{j+1}}-h_{n_j}-s(n_j,p_{n_j}).$ Assume now that for infinity many primes $q$, $m_j \not \equiv 0$, for infinity many $j$. Then, for infinity many primes $p \neq q$, we have 
	$$\widehat{\alpha_{(q)}}(pm_j \pm pm_k)=\widehat{\alpha_{(q)}}(pm_j)=\widehat{\alpha_{(q)}}(pm_k)=0.$$
	We have also
	\begin{enumerate}
		\item $\widehat{\alpha_{(p)}}(\pm p m_j)=\frac{1}{p_{n_j}},$
		\item $\widehat{\alpha_{(p)}}(p m_j\pm p m_k)=\widehat{\alpha_{(p)}}(pm_j) \widehat{\alpha_{(p)}}(pm_k)$, $\forall j \neq k$. 
	\end{enumerate}
	Whence, by taking $\phi_{n}(z)=z^{p m_n}$, we get 
	\begin{align}
	&\sum_{ k \geq 0} \sup_{n \in \N^*}\Big\{\Big|\int \phi_{n}(x) \phi_{n+k}(x) d\alpha_{(p)}- 
	\int \phi_{n}(x) d\alpha_{(p)} \int \phi_{n+k}(x) d\alpha_{(p)}\Big|\Big\}<\infty. \\
	&\sum_{ k \geq 0} \sup_{n \in \N^*}\Big\{\Big|\int \phi_{n}(x) \phi_{n+k}(x) d\alpha_{(q)}- 
	\int \phi_{n}(x) d\alpha_{(q)} \int \phi_{n+k}(x) d\alpha_{(q)}\Big|\Big\}<\infty.\\
	& \sum_{ n = 1}^{+\infty}\Big|\int \phi_{n}(x)  d\alpha_{(p)}-\int \phi_{n}(x)  d\alpha_{(q)}\Big|^2=+\infty. 
	\end{align}
	Whence 
	$$\frac{d\alpha_{(p)}}{d\alpha_{(q)}}=0.$$
	Applying the similar arguments as in Corollary 5.3 in \cite{elabdal-Nad}, we conclude that 
	$$\frac{d\mu_{(p)}}{d\mu_{(q)}}=0.$$
	That is, for infinitely many primes $p \neq q,$ we have 
	$$H(\mu_{(p)},\mu_{(q)})=0.$$
	This completes the proof of the theorem.
\end{proof}
%\todoH{Add Peryri\'ere as formulated by Klemes-Reinhold}
In the next subsection, we present the arguments based on the disjointedness of the powers.

\paragraph{\bf M\"obius disjointness and spectral disjointness of the powers.}
As before, let us observe that we have the following 
\[
\Big| \frac1{N}\sum_{j=1}^{N}f(T^{pn}x)f(T^{qn}x) \Big| \leq
\bigintss_{\T}  \Big|\frac{1}{\sqrt{N}}\sum_{j=1}^{N}f(T^{pn}x) e^{i n \theta}\Big|
\Big|\frac{1}{\sqrt{N}}\sum_{j=1}^{N}f(T^{qn}x) e^{i n \theta}\Big| d\theta,
\]
Hence,
\[
\Big| \frac1{N}\sum_{j=1}^{N}f(T^{pn}x)f(T^{qn}x) \Big| \leq
\frac{N}{N}H(\sigma_{f,p,N},\sigma_{f,q,N})
\]
Which means that
\[
\Big| \frac1{N}\sum_{j=0}^{N-1}f(T^{pn}x)f(T^{qn}x) \Big| \leq
H(\sigma_{f,T^p,N},\sigma_{f,T^q,N})
\]
and, by Coquet-Mand\'es-France-Kamae lemma \ref{CoquetMFK}, we deduce that we have
\[
\limsup\Big| \frac1{N}\sum_{j=0}^{N-1}f(T^{pn}x)f(T^{qn}x) \Big| \leq
\limsup H(\sigma_{f,T^p,N};\sigma_{f,T^q,N}) \leq H(\sigma_{f,T^p}
;\sigma_{f,T^q}).
\]
At this point, we notice that if $T^p$ and $T^q$ are spectrally disjoint for infinitely many $p\neq q$ then $T$ has singular spectrum (if not, $T$ has an absolutely continuous component then $T^m$ has a Lebesgue component for $m\geq m_0$). It follows that the M\"{o}bius disjointness holds by el Abdalaoui-Nerurkar's recent result \cite{AM}.\\

\paragraph{\bf Remark.} 
Using our language, let us state the some known results in the situation of uniquely ergodic topological model of finite rank-one maps $(X,\A,T,P)$ ($P$ is a probability measure) where all the parameters are uniformly bounded. We have to be careful here since the symbolic dynamical system in the case of rank-one maps with spacers added on the last column is not uniquely ergodic. The proof given in \cite{elabdal-lem-de-la-rue} covers this latter case. 
%and it is inspired by the proof given by Bourgain-Sarnak-Ziegler \cite{BSZ} of their theorem on the disjointness of M\"obius from horocycle flows. 
We would like to add also that Bourgain's method and therefore our proofs of Theorem \ref{Bourgain} apply to infinite symbolic rank-one maps for which the M\"{o}bius disjointness holds. In fact, the proof gives that if  $f$ be a continuous function with zero mean (i.e. $\displaystyle \int f(y) d\Pro=0$) then 
$$\limsup H(\sigma_{f,(p),N},\sigma_{f,(q),N})=0.$$
Later, el Abdalaoui-Lema\'nczyk-de-la-Rue \cite{elabdal-lem-de-la-rue} proved that for the case of some class of finite rank-one maps which include the case of of finite rank-one maps with bounded parameters we have, for a large prime $p \neq q$, 
\begin{align}
H(\sigma_{f,T^p},\sigma_{f,T^q})=0.
\end{align}
The proof is based on the computation of the weak-limit of the maps.

In the same year, V. V. Ryzhikov \cite{Ryzhikov} established  that for the finite weak-mixing rank-one maps  with bounded parameters  we have the following weaker version:
for a large prime $p \neq q$, 
\begin{align}
\limsup\Big| \frac1{N}\sum_{j=0}^{N-1}f(T^{pn}x)f(T^{qn}x) \Big| =0.
\end{align}
This follows from the fact that those maps have a minimal self-joinings.

%\paragraph{\bf Bourgain Theorem (5-12-11): $\limsup H(\sigma_{f,(p),N},\sigma_{f,(q),N})=0$ }
%\paragraph{\bf Ryzhikov Theorem (12-12-12): $\limsup\Big| \frac1{N}\sum_{j=0}^{N-1}f(T^{pn}x)f(T^{qn}x) \Big| =0$}
%\paragraph{\bf elAbdalaoui-Lema\'nczyk-de-la-Rue Theorem(21-09-12):$H(\sigma_{f,T^p},\sigma_{f,T^q})=0$}
%\begin{rem}One can adapted the Bourgain method combined with the Riesz products machinery to show that the M\"obius function is disjoint from the rank-one maps for which the cutting parameter $p_n$ satisfy
%	\[
%	\sum_{n=1}^{+\infty}\frac1{p_n^2}=+\infty.
%	\]
%\end{rem}

%\end{comment}
%%%%%%

\begin{thank}
The first author gratefully thanks Rutgers University for an invitation to visit in 2018 and he would like to thank Mahesh Nerurkar for many discussions and Zoom-discussion on the subjects. The second-named author would like to thank the University of Rouen Normandy for an invitation to visit in 2017 when the work on this paper was started. 
\end{thank}

\end{document}